\documentclass[12pt]{amsart} 
\usepackage{amsmath,amssymb,amsthm,amsfonts,mathrsfs,bm}
\usepackage[all]{xy}
\usepackage{dsfont}
\usepackage{hyperref}
\usepackage{color}
\usepackage{esint}
\usepackage{mathtools}
\mathtoolsset{showonlyrefs}
\usepackage{slashed}

\usepackage{tikz-cd}

\usepackage[margin=1.1in]{geometry}

\usepackage[obeyDraft]{todonotes}

\newcommand{\p}{\partial}
\newcommand{\R}{\mathbb{R}}

\newcommand{\D}{\slashed{D}}

\newcommand{\dd}{\mathop{}\!\mathrm{d}}
\newcommand{\snabla}{\slashed{\nabla}} 

\newcommand{\cliff}{\mathfrak{m}}

\newcommand{\Abracket}[1]{\left<#1\right>} 
\newcommand{\parenthesis}[1]{\left(#1\right)} 
\newcommand{\braces}[1]{\left\{#1\right\}} 

\newcommand{\eps}{\varepsilon}

\newcommand{\bH}{\mathbb{H}}

\newcommand{\bZ}{\mathbb{Z}}

\DeclareMathOperator{\Cl}{Cl}

\DeclareMathOperator{\dv}{\dd{vol}}

\DeclareMathOperator{\End}{End}

\DeclareMathOperator{\Hess}{Hess}

\DeclareMathOperator{\id}{Id}
\DeclareMathOperator{\Ker}{Ker}

\DeclareMathOperator{\Spect}{Spect}

\DeclareMathOperator{\Vol}{Vol}

\newtheorem{thm}{Theorem}[section]

\newtheorem{lemma}[thm]{Lemma}
\newtheorem{prop}[thm]{Proposition}

\newtheorem{rmk}[]{Remark}

\title{Existence results for a super Toda system}

\author[A. Jevnikar]{Aleks Jevnikar}
\address{Aleks Jevnikar, Department of Mathematics, Computer Science and Physics, University of Udine, Via delle Scienze 206, 33100 Udine, Italy.}
\email{aleks.jevnikar@uniud.it}

\author[R. Wu]{Ruijun Wu}
\address{Ruijun Wu, SISSA, Via Bonomea, 265, 34136 Trieste, Italy}
\email{ruijun.wu@sissa.it}

\begin{document}
\begin{abstract}
 We solve a super Toda system on a closed Riemann surface of genus~$\gamma>1$ and with some particular spin structures. 
 This generalizes the min-max methods and results for super Liouville equations and gives new existence results for super Toda systems. 
\end{abstract}

\maketitle

{\footnotesize
\emph{Keywords}: super Toda systems, existence results, min-max methods.

\medskip

\emph{2010 MSC}: 58J05, 35A01, 58E05, 81Q60.}

\

\section{Introduction}

In the present work we are concerned with the existence of solutions of a super Toda system. 
This topic is motivated by the recent blowup analysis on super Toda system on Riemann surfaces~\cite{jost2019superToda}, and from the physics side by~\cite{olshanetsky1983susy} and more recently by~\cite{alfaro2010multi}. 
Briefly speaking, the equations on dilaton gravity and gauge theory take the form of Toda equations, and their supersymmetric extensions in quantum field theory takes the form of super Toda systems. 
The physical literature emphasizes the Lie algebraic structure of the equations and special solution in flat spacetime, while here we treat it variationally and try to obtain some new existence results of nontrivial classical solutions with curved metrics. 
The system has a variational structure with a strongly indefinite functional, and we locate it on a general Riemann surface with a restriction of genus greater than one. 

\medskip

To formulate our result, let us start from a variational viewpoint.
We will not restrict ourselves to the physics-relevant applications and hence we omit the physics dimensional constants to make the mathematical structure more clear.
Moreover, the Lie group of the theory will be~$SU(N+1)$. However, the result holds for a much wider class of problems, see the discussion after Theorem~\ref{thm:existence}. For simplicity we will restrict to the case~$N=2$, since the case for general~$N\ge 2$ is similar but only notationally heavy.
Note that for~$N=1$ this reduces to the super Liouville equations, which was considered in~\cite{jevnikar2020existence,jevnikar2021existence} and shows some new features in variational problems involving Dirac operators.

For what concerns spin geometry and the Dirac operator we refer to Section 2 and the references therein. Let~$(M,g)$ be a Riemann surface equipped with a smooth Riemannian metric. 
With a given spin structure, let~$\Sigma_g M$ denote the associated spinor bundle. 
A field in the model is described by a tuple~$(\mathbf{u},\bm{\psi})$, where
\begin{align}
 \mathbf{u}=&(u_1,\cdots, u_N), \quad \mbox{ with } u_j\in C^\infty(M),\; (1\le j\le N) \\
 \bm{\psi}=&(\psi_1,\cdots,\psi_N), \quad \mbox{ with } \psi_j\in \Gamma(\Sigma_g M), \; (1\le j\le N). 
\end{align}
Were the spinors anti-commuting variables, there would be a supersymmetry transformation between the bosonic fields~$\bm{u}$ and the fermionic fields~$\bm{\psi}$. 
The analysis for anti-commuting variables are far from available and thus we have to take the spinors as commuting variables, in which case we can analyze the model variationally, still maintaining most of the properties. 
The super Toda system under consideration takes the form: for~$1\le j\le N$,
\begin{align}
 \Delta u_j =&\sum a_{jk}\parenthesis{e^{2u_j}- e^{u_j}|\psi_j|^2} - K_g, \\
 \D\psi_j= & -e^{u_j}\psi_j, 
\end{align}
where~$A=(a_{jk})$ is the Cartan matrix for~$SU(N+1)$, that is
\begin{align}
 A=(a_{jk})
 =\begin{pmatrix}
   2 & -1 & 0 & \cdots & \cdots & 0& 0 \\
   -1 & 2 & -1& \cdots & \cdots & 0& 0 \\
   0 & -1 & 2 & \cdots & \cdots & 0& 0 \\
   \vdots & \vdots & \vdots &  &  & \vdots  &\vdots \\
   0 & 0 & 0& \cdots & \cdots & -1 & 2
  \end{pmatrix},
\end{align}
$K_j$ is the Gaussian curvature of $M$ and $\D$ is the Dirac operator acting on spinors $\psi_j$. 

Observe that there is a different sign convention with respect to the one considered in~\cite{jost2019superToda}. If we take into the consideration the topology and the curvature of the underlying surface, which we have assumed to have genus bigger than one, we are naturally led into the above system. This is consistent with the case~$N=1$, i.e. super Liouville  equations, and the prescribed curvature problem which is in turn tight with the Gauss--Bonnet formula, see the discussion in~\cite{jevnikar2020existence}. The case of genus one or zero should be formulated with a different notation.

Since the Dirac operator has an unbounded symmetric spectrum~\cite{ginoux2009dirac}, the system is strongly indefinite, and the indefinite part is more coupled in the system case; this is the main difficulty that we have to deal with. 

\medskip

From now on, we will restrict to the case~$N=2$ to simplify the notation, hence for us:
\begin{align}
 A= \begin{pmatrix}
     2 & -1 \\
     -1 & 2
    \end{pmatrix}
\end{align}
and the configuration space is the Hilbert space
\begin{align}
 H\equiv (H^1(M))^2 \times \parenthesis{H^{\frac{1}{2}}(\Sigma_g M)}^2
\end{align}
consisting of tuples~$(\bm{u},\bm{\psi})$ with Sobolev regularity, see Section 2 for more details. 
With a positive parameter~$\rho\in \R_+$, we consider the action functional
\begin{align}
 J_\rho\colon H\to\R
\end{align}
defined by 
\begin{align}\label{eq:functional}
 J_\rho(\bm{u},\bm{\psi}) \coloneqq 
 &\int_M \frac{1}{3}\parenthesis{|\nabla u_1|^2 +|\nabla u_2|^2+\Abracket{\nabla u_1, \nabla u_2}} +2 K_g (u_1 + u_2) + (e^{2u_1}+ e^{2u_2}) \\
  & \qquad\qquad 
       + \Abracket{\D\psi_1- \rho e^{u_1}\psi_1,\psi_1}
          +\Abracket{\D\psi_2-\rho e^{u_2}\psi_2,\psi_2} \dv_g
        +\int_M 2K_g \dv_g. 
\end{align}
The last additive constant is just to ensure~$J_\rho(0,0)=0$, without any affection to the variational structure. 
This is a smooth functional, and the Euler--Lagrange equations are given by 
\begin{equation}\label{eq:super Toda-1} \tag{sT}
 \begin{cases}
  \frac{1}{3}(-2\Delta u_1- \Delta u_2)+ 2K_g + 2 e^{2u_1} -\rho e^{u_1} |\psi_1|^2 =0, \vspace{0.2cm}\\
  \frac{1}{3}(-\Delta u_1-2\Delta u_2) + 2K_g + 2 e^{2u_2} -\rho e^{u_2} |\psi_2|^2 =0, \vspace{0.2cm}\\
  \D\psi_1 = \rho e^{u_1} \psi_1, \vspace{0.2cm}\\
  \D\psi_2 = \rho e^{u_2} \psi_2, 
 \end{cases}
\end{equation}
which is equivalent to
\begin{align}\label{eq:super Toda-2} \tag{$sT'$}
 \begin{cases}
  \Delta u_1=2\parenthesis{2e^{2u_1}- e^{2u_2} }   
    -\rho\parenthesis{2e^{u_1}|\psi_1|^2 - e^{u_2}|\psi_2|^2 }+ 2K_g,\vspace{0.2cm}\\ 
  \Delta u_2=2\parenthesis{2 e^{2u_2}-e^{2u_1}} 
   -\rho\parenthesis{2 e^{u_2}|\psi_2|^2-e^{u_1}|\psi_1|^2}+ 2K_g, \vspace{0.2cm}\\
   \D\psi_j = \rho e^{u_j}\psi_j, \qquad \qquad j=1,2.
 \end{cases}
\end{align}
These are known as the super Toda system which we aim to solve.
The system is conformal invariant, i.e.: for a conformal metric~$\widetilde{g}= e^{2w}g$ with~$w\in C^\infty(M)$, let
\begin{align}
 &\widetilde{u}_j\coloneqq u_j -w, \\
 &\widetilde{\psi}_j\coloneqq e^{-\frac{w}{2}}\beta(\psi_j), \quad j=1,2
\end{align}
where~$\beta\colon \Sigma_g M \to \Sigma_{\widetilde{g}} M$ is the induced isometric isomorphism between the spinor bundles associated to the different metrics~$g$ and~$\widetilde{g}$ respectively~\cite{ginoux2009dirac, jost2018symmetries}, then 
\begin{align}
 &K_{\widetilde{g}}= e^{-2w}(K_g- \Delta_g w), \ \
 &\Delta_{\widetilde{g}} u_j= e^{-2w} \Delta_{g} u_j, \quad j=1,2
\end{align}
and 
\begin{align}
 \D_{\widetilde{g}} \widetilde{\psi}_j = e^{-\frac{3}{2}w} \beta(\D\psi_j) = \rho e^{\widetilde{u}_j}\widetilde{\psi}_j, \qquad j=1,2.
\end{align}
As a consequence,~$(\widetilde{\bm{u}}, \widetilde{\bm{\psi}})$ satisfies the system of the same form as~\eqref{eq:super Toda-2} with respect to the metric~$\widetilde{g}$. 
Thus, without loss of generality, we can assume that we start with a metric~$g$ with constant Gaussian curvature~$K_g\equiv -1$. Then,
note that there is a trivial solution:~$\bm{u}=(0,0)$,~$\bm{\psi}=(0,0)$. 

\medskip

If one neglects supersymmetry and physics relevance, there could be other ways to modify the super Toda equations to locate it on a closed surface. 
For instance, one can consider equations of mean field type, see~\cite{malchiodi2007some} and the references therein. 
However, the form here is better suited for the inclusion of spinors. 

\begin{rmk}
In vectorial form,~\eqref{eq:super Toda-2} can also be written as 
\begin{equation}
 \begin{cases}
  \Delta \begin{pmatrix}u_1 \\ u_2 \end{pmatrix}
  = & 2\begin{pmatrix} 2 & -1 \\ -1 &2 \end{pmatrix}
      \begin{pmatrix} e^{2u_1}\\ e^{2u_2} \end{pmatrix}
      + 2K_g\begin{pmatrix} 1 \\ 1\end{pmatrix} 
      -\rho\begin{pmatrix} 2 & -1 \\ -1 & 2\end{pmatrix}
      \begin{pmatrix} e^{u_1}|\psi_1|^2 \\ e^{u_2}|\psi_2|^2 \end{pmatrix},  \\
  \D\begin{pmatrix}\psi_1 \\ \psi_2 \end{pmatrix}
  =& \rho \begin{pmatrix} e^{u_1} & 0 \\ 0 & e^{u_2} \end{pmatrix}
       \begin{pmatrix} \psi_1 \\ \psi_2 \end{pmatrix}. 
 \end{cases}
\end{equation}
This can be somehow viewed as 
\begin{equation}
 \begin{cases}
  \Delta\bm{u} = 2A e^{2\bm{u}} +2 K_g -\rho A e^{\bm{u}} |\bm{\psi}|^2, \\
  \D\bm{\psi}= \rho e^{\bm{u}} \bm{\psi}
 \end{cases}
\end{equation}
and can be regarded as a generalization of the super Liouville system, see~\cite{jevnikar2020existence}. 
 
\end{rmk}

\medskip

In the present work we will use minimax method to find a nonzero solution on closed Riemann surface~$M$ with~$genus(M)=\gamma>1$. 
The basic idea is to locate the functional on a Nehari type manifold with infinite codimension, where the functional has only finite index at the trivial critical point~$(\bm{0}, \bm{0})\in H$. 
A similar argument was applied to the Dirac-Einstein system~\cite{maalaoui2019characterization} and the super Liouville problem~\cite{jevnikar2020existence}, noting that the latter is actually a special case of the super Toda system. 
Here we extend the methodology to systems, which involves some technical nontriviality. In particular, some extra care is devoted to the study of the Palais-Smale condition in order to control the coupling effect of the different fields in play.

\medskip

We call~$\rho\in\R_+$ \emph{exceptional} if~$\rho\in \Spect(\D_g)$, and we say that a solution~$(\bm{u},\bm{\psi})$ is nontrivial if it is nonzero. 
Note that if~$u_1$ and~$u_2$ are constant functions, then~$\psi_1$ and~$\psi_2$ are eigenspinors of constant length, which is impossible by Efimov theorem since we assume~$K_g\equiv -1$, see~\cite{efimov1968hyperbolic, jevnikar2021sinhgordon,  milnor1972efimov}. 

\begin{thm}\label{thm:existence}
 Let~$(M,g)$ be a closed surface with genus~$\gamma>1$. 
 Then, there exists a spin structure such that~\eqref{eq:super Toda-1} is solvable for any~$\rho$ non-exceptional.  
\end{thm}

Actually, in general there exist many spin structures satisfying the above result. As it will be pointed out in the proof, one just needs  a spin structure with~$\Ker\D=\{0\}$, for which we refer to~\cite{bar1992harmonic, bures1994harmonic} or Section 2 in~\cite{jevnikar2020existence}.

\medskip

Besides the known solutions in flat domains, this seems to be the first existence result of nontrivial solutions on a general closed surface.
Once there is a solution with nonzero spinors, there corresponds a 3-dimensional family of solutions because of the quaternionic structure on the spinor bundle (we are taking the Clifford algebra~$\Cl(2,0)$). 
Solutions in this 3-dimensional family are geometrically equivalent. 
We do expect that there exist other geometrically inequivalent solutions. 

\medskip

The method is much more robust than this. First, one can consider the general case~$N	\geq2$. Second, we can allow the coefficient matrix to be diagonally dominant and thus different Lie algebras can be treated other than~$SU(N+1)$. 
Moreover, we can consider~$\bm{\rho}=(\rho_1,\cdots, \rho_N)$ instead of a single parameter~$\rho$, as long as each of~$\rho_j$ is positive.   
Finally, if we start with a surface with constant Gaussian curvature (say~$K_g\equiv -1$), then we may even give up the requirement of conformal invariance, though the conformal invariance is relevant in physics.
The potential of exponential type, on one hand, does not go to infinity in~$L^p$ sense for~$p<\infty$ thanks to Moser-Trudinger inequality in dimension two; on the other hand, it may allow~$u$ to decrease to~$-\infty$ which will get out of control. 
This is then reflected mainly in the study of the Palais-Smale condition. 
Indeed, when the potential is replaced by a coercive one, for example as in the sinh-Gordon case, we could get better estimate in the spirit of~\cite{jevnikar2021sinhgordon} and hence we can allow the Dirac operator to have nontrivial kernel, that is, we can allow arbitrary spin structures.

\medskip

The paper is organized as follows.
Section 2 contains some brief preliminaries about the spectral properties of the Dirac operator and the fractional Sobolev spaces of spinors as well as a quick recall of Moser--Trudinger embedding in 2D. 
In Section 3 we introduce a suitable Nehari manifold, on which is constrained functional is shown to satisfy the Palais--Smale condition. 
This enables us to search for solutions of saddle type using min-max method, which is done in Section 4.

\

\section{Preliminaries}

Here we collect some basic facts concerning the spectral properties of the Dirac operator, the fractional Sobolev space of spinors, and the Moser-Trudinger embeddings of functions. 
Since these materials are rather well-known, we will be sketchy, and the reader can refer to~\cite{aubin1998some, ginoux2009dirac, jevnikar2020existence, jost2011riemannian, lawson1989spin} for more information. 

\subsection{Spectral properties of Dirac operators on surfaces}
Let~$M$ be a closed Riemann surface of genus~$\gamma>1$ and let~$g$ be a Riemannian metric in the given conformal class. 
Choose a spin structure on~$M$ and let~$\Sigma_g M$ (abbreviated as~$\Sigma M$) be the associated spinor bundle, with fiber~$\bH\cong\R^4$. 
The spinor bundle is equipped with a fiberwise real inner product~$\Abracket{\cdot,\cdot}$, a canonical spin connection~$\snabla$, and a Clifford map~$\cliff\colon TM\to \End(\Sigma M)$ satisfying the Clifford relation 
\begin{align}
 \cliff(X)\cliff(Y)+\cliff(Y)\cliff(X)= - 2g(X,Y), \qquad \forall X, Y \in \Gamma(TM). 
\end{align}
Up to the identification~$T^*M \cong TM$ by the Riemannian metric, we can view the Clifford map as~$\cliff\colon T^*M\otimes \Sigma \to \Sigma M$. 
Then the Dirac operator on~$\Sigma M$ is defined as
\begin{align}
 \D\coloneqq \cliff \circ \snabla \colon  \Gamma(\Sigma M) \to \Gamma(\Sigma M). 
\end{align}
It is an essentially self-adjoint operator which is elliptic on the closed surface~$M$, hence has a real spectrum consisting of eigenvalues, and the eigenvalues are unbounded in both directions.
More precisely, let~$\lambda_j$, $j\in \bZ_*=\bZ\setminus \{0 \}$ be the nonzero eigenvalues counted with multiplicities and in a non-decreasing order, then
\begin{align}
  -\infty \leftarrow\cdots\le \lambda_{-k-1}\le \lambda_{-k}\le\cdots\le \lambda_{-1}\le 0 
 \le \lambda_1\le \cdots \le \lambda_k \le \lambda_{k+1}\le \cdots \to +\infty;
\end{align}
the Dirac operator may also have kernels, necessarily finite dimensional. 
From now on we assume that the spin structure is chosen such that~$\Ker\D=\{0\}$, which is always possible, see~\cite{bar1992harmonic, bures1994harmonic} or Section 2 in~\cite{jevnikar2020existence}. 
We denote the eigenspinors by~$\Psi_j$,~$j\in \bZ_*$ which form a complete ~$L^2-$orthonormal basis for the~$L^2$ spinors.
Note that the eigenspinors are smooth by classical elliptic regularity theory. 

\medskip

There are two types of special structures on the spinor bundle: the volume element endomorphism and the quaternionic structures. 
The volume element endomorphism, or known as the chirality operator, is~$\omega=\cliff(\dv_g)$ where we have extended~$\cliff$ to the exterior differential forms. 
In terms of an oriented locally~$g$-orthonormal frame~$(e_1, e_2)$, \begin{align}
 \omega=\cliff(e_1)\cliff(e_2)\in \End(\Sigma M), & & 
 \mbox{ such that } \quad \omega^2= -\id.
\end{align}
That is,~$\omega$ defines an almost complex structure on~$\Sigma M$. 
Another family of almost complex structures on~$\Sigma M$ come from the quaternionic structure of the spinor bundle, which form a three-dimensional family~$\mathcal{J}$ such that for each~$\mathbf{j}\in \mathcal{J}$, 
\begin{align}
 \mathbf{j}^2=-\id, & & 
 \D(\mathbf{j}(\psi))= \mathbf{j}(\D\psi). 
\end{align}
Therefore, each eigenspace of~$\D$ admits a quaternionic structrue and has real dimension a multiple of four. 
In contrast, the volume element anti-commute with the Dirac operator
\begin{align}
 \D(\omega\cdot \psi)= -\omega\cdot\D\psi
\end{align}
which implies that the spectrum~$\Spect(\D)$ is symmetric with respect to the origin:~$\lambda_{-k}= -\lambda_k$. 

\medskip

\subsection{The fractional Sobolev space for spinors}
We take the definitions and basic properties of fractional Dirac operators, and fractional Sobolev spaces for spinors, from~\cite{ammann2003habil}, and will mostly use the notation from~\cite{jevnikar2021sinhgordon}. 

\medskip

A spinor~$\psi\in L^2(\Sigma M)$ is expressed in terms of the orthonormal basis~$\{\Psi_k\}_{k\in \mathbb{Z}_*}$ as 
\begin{align}
 \psi= \sum_{j\in \bZ_*} a_j \Psi_j,
    & & 
    \mbox{ with }
    \quad 
    \|\psi\|_{L^2}^2 = \sum_j |a_j|^2. 
\end{align}
For~$\psi \in C^1$ (or in~$H^1=W^{1,2}$), we can differentiate term by term and get
\begin{align}
 \D\psi =\sum_{j\in\bZ_*} a_j \lambda_j \Psi_j. 
\end{align}
Then, by the~$L^2$ orthonormality of the basis, 
\begin{align}
 \|\D\psi\|_{L^2}^2 =\sum_j |a_j|^2 \lambda_j^2, 
\end{align}
as long as the right hand side is finite. 
The fractional Dirac operator on spinors and the fractional Sobolev spaces for spinors are defined in a similar way.
That is, for any~$s>0$, define~$|\D|^s\colon \Gamma(\Sigma M)\to \Gamma(\Sigma M)$ by 
\begin{align}
 |\D|^s \psi 
 = \sum_{j\in \bZ_*} |\lambda_j|^s a_j\Psi_j, 
\end{align}
and the space of~$H^s$ spinors is
\begin{align}
 H^s(\Sigma M)\coloneqq 
 \braces{ \psi\in L^2(\Sigma M) \mid 
  \Abracket{\psi,\psi}_{H^s}<\infty}. 
\end{align}
where 
\begin{align}
 \Abracket{\psi,\phi}_{H^s} \coloneqq 
 \Abracket{\psi,\phi}_{L^2} + \Abracket{|\D|^s\psi, |\D|^s\phi}_{L^2}
\end{align}
is an inner product, making~$H^s(\Sigma M)$ a Hilbert space. 
For~$s<0$ we define~$H^s(\Sigma M)$ as the dual space of~$H^{-s}(\Sigma M)$ as usual. 
There is no confusion of~$H^k(\Sigma M)$ for~$k\in \mathbb{N}$ since in this case the above Sobolev space coincides with the classical Sobolev space of spinor $W^{s,2}(\Sigma M)$ defined via covariant derivatives, though~$|\D|^k\neq \D^k$ (since they differ on the eigenspinors associated to negative eigenvalues by a sign).  
The spinorial version of Sobolev embeddings have the same form as for functions. In particular for~$s\in (0,1)$, the space~$H^s(\Sigma M)$ continuously embeds into~$L^q(\Sigma M)$ for~$1\le q\le\frac{2}{1-s}$ and compactly embeds in~$L^q(\Sigma M)$ for~$1\le q<\frac{2}{1-s}$. 
What will be concerned is the case~$s=\frac{1}{2}$, for which we have the continuous embedding
\begin{align}
 H^{\frac{1}{2}}(\Sigma M)\hookrightarrow L^4(\Sigma M). 
\end{align}
This is a suitable space of Sobolev spinors to work with in the variational setting. 

\medskip

According to the signs of the eigenvalues, we have the following decomposition
\begin{align}\label{eq:splitting of spinor space-I}
 H^{\frac{1}{2}}(\Sigma M) = H^{\frac{1}{2},+}(\Sigma M) \oplus H^{\frac{1}{2},-}(\Sigma M),
 \qquad \psi= \psi^+ +\psi^-,
\end{align}
where~$H^{1/2,\pm}(\Sigma M)$ denotes the closure of the subspaces spanned by eigenspinors corresponding to positive resp. negative eigenvalues. 
With an abuse of terminology, we will call~$\psi^+$ resp.~$\psi^-$ the positive resp. negative components of the spinor~$\psi$. 
Let~$P^\pm$ denote the projection operator:~$P^\pm\psi=\psi^{\pm}$. 
Given a parameter~$\rho \notin\Spect(\D)$ with~$\rho>0$, for later convenience we further split the space~$H^{1/2,+}(\Sigma M)$ into 
\begin{align}
 H^{\frac{1}{2},+}(\Sigma M)
 = H^{\frac{1}{2},+}_a(\Sigma M)\oplus H^{\frac{1}{2},+}_b(\Sigma M) 
\end{align}
with~$H^{1/2,+}_a(\Sigma M)$ resp.~$H^{1/2,+}_b(\Sigma M)$ being the~$H^{1/2}$-closure of the subspaces spanned by eigenspinors corresponding to positive eigenvalues ``above'' respectively ``below''~$\rho$. 
Therefore,~\eqref{eq:splitting of spinor space-I} can be refined into
\begin{align}\label{eq:splitting of spinor space-II}
 H^{\frac{1}{2}}(\Sigma M)
 = H^{\frac{1}{2},+}_a(\Sigma M) 
 \oplus H^{\frac{1}{2},+}_b (\Sigma M)
 \oplus H^{\frac{1}{2},-}(\Sigma M),
 & & \mbox{ and } \quad \psi= \psi^+_a +\psi^+_b +\psi^-. 
\end{align}

\medskip

\subsection{Morser-Trudinger embeddings in 2D}
The suitable space of functions here is the classical Sobolev space~$H^1(M)=W^{1,2}(M,\R)$, which comes with the decomposition
\begin{align}
 H^1(M)=\R\oplus H^1_0(M), \qquad u= \bar{u}+ \widehat{u}
\end{align}
where~$\bar{u}=\fint_M u\dv_g $ is the average of~$u$, and~$\widehat{u}$ denotes the oscillating part with zero average. 
By Poincar\'e's inequality,~$\|\nabla \widehat{u}\|_{L^2}$ defines a norm equivalent to~$\|\widehat{u}\|_{H^1}$ on~$H^1_0(M)$, and 
\begin{equation}
 |\bar{u}|+\|\nabla\widehat{u}\|_{L^2}
\end{equation}
is a norm equivalent to~$\|u\|_{H^1}$. 
The Sobolev embedding theorems imply that for any~$p<\infty$, ~$H^1(M)$ embeds into~$L^p(M)$ continuously and compactly. 
Furthermore, the Moser--Trudinger inequality implies that~$e^{u}$ is~$L^p$ integrable for any~$p\in [1,+\infty)$ and the mapping
\begin{equation}
 H^1(M)\ni u\mapsto e^u\in L^1(M)
\end{equation}
is compact (see e.g.~\cite[Theorem 2.46]{aubin1998some}). 
Consequently the mappings~$H^1(M)\ni u\mapsto e^u\in L^p(M)$ are compact for all~$p \ge 1$.  
This will be crucial for the compactness of the problem, namely this compactness of Moser-Trudinger embedding essentially guarantees that validity of the Palais--Smale condition, as will be clear later. 

\

\section{A Nehari manifold and Palais-Smale conditions}

Since the functional~$J_\rho$ on the Hilbert space~$H$ is strongly indefinite, and the nonlinearity is exponential in the negative direction, it is desirable to restrict it to a nice Nehari-type submanifold of~$H$, so that we can get better control in the negative part of the functional, as in~\cite{szulkin2009ground, szulkin2010themethod}. 
This can be done in the spirit of~\cite{jevnikar2020existence,jevnikar2021sinhgordon}. 
Namely, consider the map
\begin{align}
 G=(G_1, G_2)\colon \parenthesis{H^1(M)}^2 \times \parenthesis{H^{\frac{1}{2}}(\Sigma M)} \to \parenthesis{H^{\frac{1}{2},-}(\Sigma M)}^2 
\end{align}
given by 
\begin{align}
  G_j(u_j,\psi_j)\coloneqq 
 P^-(1+|\D|)^{-1}(\psi_j -\rho e^{u_j}\psi_j)=0, \; j=1,2,
\end{align}
and define 
\begin{align}
 N_\rho\coloneqq G^{-1}(0,0) 
 =\biggr\{(\bm{u},\bm{\psi})\in H \mid  G_j(u_j,\psi_j)
 =P^-(1+|\D|)^{-1}(\psi_j -\rho e^{u_j}\psi_j)=0, \; j=1,2\biggr\}. 
\end{align}

Then, we have
\begin{lemma}
 $N_\rho$ is a submanifold of~$H$. 
\end{lemma}
\begin{proof}
 It suffices to show that~$\dd G(\bm{u},\bm{\psi})\colon H\to \parenthesis{H^{1/2,-}(\Sigma M)}^2$ is surjective. 
 Note that 
 \begin{align}
  \dd G(\bm{u},\bm{\psi})[\bm{v},\bm{\phi}]
  =\bigr( \dd G_1(u_1,\psi_1)[v_1, \phi_1],  \dd G_2(u_2,\psi_2)[v_2, \phi_2]\bigr)
 \end{align}
 where
 \begin{align}
  \dd G_j(u_j,\psi_j)[v_j,\phi_j]
  =P^- (1+|\D|)^{-1} \parenthesis{\D\phi_j-\rho e^{u_j}\phi_j - \rho e^{u_j} v_j \psi_j}, \quad j=1,2.
 \end{align}
 Therefore, for any~$\bm{\phi}\in \parenthesis{ H^{\frac{1}{2},-}(\Sigma M) }^{2}$, we have 
 \begin{multline}
  \Abracket{\dd G(\bm{u},\bm{\psi})[0,\bm{\phi}],\bm{\phi}}_{(H^{1/2}(\Sigma M))^2} 
  = \int_M \Abracket{\D\phi_1-\rho e^{u_1}\phi_1,\phi_1} 
     +\Abracket{\D\phi_2-\rho e^{u_2}\phi_2,\phi_2}\dv_g  \\
  = \int_M \Abracket{\D\phi_1,\phi_1}\dv_g 
    -\rho\int_M e^{u_1}|\phi_1|^2\dv_g 
    +\int_M \Abracket{\D\phi_2,\phi_2}\dv_g 
    -\rho\int_M e^{u_2}|\phi_2|^2\dv_g
 \end{multline}
 which defines a nondegenerate quadratic form on the Hilbert space~$\parenthesis{H^{1/2,-}(\Sigma M)}^2$ and this quadratic form has bounded norm as a bilinear form at each~$(\bm{u},\bm{\psi})\in H$. 
 Thus~$\dd G(\bm{u},\bm{\psi})$ is a surjective map. 
\end{proof}

Moreover, it holds
\begin{lemma}
 $N_\rho\subset H$ is a Nehari manifold for~$J_\rho$. 
\end{lemma}
\begin{proof}
Consider the constrained functional~$J_\rho|_{N_\rho}$. 
If~$(\bm{u},\bm{\psi})\in N_\rho$ is a constrained critical point for~$J_\rho|_{N_\rho}$, then 
\begin{align}
 \dd J_\rho(\bm{u},\bm{\psi}) = \dd G(\bm{u},\bm{\psi})[0, (\varphi_1, \varphi_2)]
\end{align}
for some~$\varphi_1, \varphi_2 \in H^{1/2,-}(\Sigma M)$, possibly depending on~$(\bm{u},\bm{\psi})$, which are actually the Lagrange multipliers. 
That is, for any~$(\bm{v},\bm{\phi})\in H$, 
\begin{align}
 \dd J_\rho(\bm{u},\bm{\psi})[\bm{v},\bm{\phi}] = \Abracket{\dd G(\bm{u},\bm{\psi})[0, (\varphi_1, \varphi_2)] , (\bm{v},\bm{\phi}) }_H. 
\end{align}
Therefore the constrained Euler--Lagrange equations are
\begin{align}\label{eq:constrained EL}
 \begin{cases}
  \frac{1}{3}\parenthesis{-2\Delta u_1-\Delta u_2}+K_g + 2e^{2u_1} -\rho e^{u_1}|\psi_1|^2 
  = -\rho e^{u_1}\Abracket{\psi_1,\varphi_1}, \\
  \frac{1}{3}\parenthesis{-\Delta u_1- 2\Delta u_2}+ K_g +  2e^{2u_2}-\rho e^{u_2}|\psi_2|^2 
  =-\rho e^{u_2}\Abracket{\psi_2,\varphi_2}, \\
  2(\D\psi_1-\rho e^{u_1}\psi_1) = \D\varphi_1-\rho e^{u_1}\varphi_1, \\
  2(\D\psi_2-\rho e^{u_2}\psi_2) = \D\varphi_2-\rho e^{u_2}\varphi_2. 
 \end{cases}
\end{align}
However, if~\eqref{eq:constrained EL} holds, then we test the spinor equations against~$\varphi_j$ to get 
\begin{align}
 0=2\int_M \Abracket{\D\psi_j-\rho e^{u_j}\psi_j,\varphi_j}\dv_g
 =\int_M \Abracket{\D\varphi_j-\rho e^{u_j}\varphi_j,\varphi_j}\dv_g 
\end{align}
which implies~$\varphi_j=0$ for~$j=1,2$. 
Hence~$(\bm{u},\bm{\psi})$ actually satisfies~\eqref{eq:super Toda-1} and is a free critical point of~$J_\rho$ on~$H$. 
\end{proof}

Thus it suffices to consider the constrained functional~$J_\rho|_{N_\rho}$ and find a nonzero critical point of it. 

\medskip

Next we deal with the main part of this section, which is the key compactness property needed to run the min-max method. It is carried out by making use of test functions and the spectral decomposition.
Here we make essential use of the topological constraint of the genus greater than one and the choice of a spin structure with~$\Ker\D=\{0\}$.  Recall that we have chosen~$K_g\equiv -1$.

\begin{prop}
 $J_\rho|_{N_\rho}$ satisfies the Palais-Smale condition. 
\end{prop}

\begin{proof}
 Let~$c\in\R$ and~$(\bm{u}^n,\bm{\psi}^n)\in N_\rho$ be a $(PS)_c$ sequence for~$J_\rho|_{N_\rho}$, that is, for some~$\varphi_1^n,\varphi_2^n\in H^{1/2,-}(\Sigma M)$,
 \begin{align}\label{eq:PS:level condition}
  J_\rho(\bm{u}^n, \bm{\psi}^n)\to c, 
 \end{align}
 \begin{align}\label{eq:PS:constraint}
  P^-(1+|\D|)^{-1}(\D\psi_j^n-\rho e^{u_j^n}\psi_j^n)=0, \qquad j=1,2,
 \end{align}
 \begin{align}\label{eq:PS:u1}
  \frac{1}{3}\parenthesis{-2\Delta u_1^n-\Delta u_2^n}+ 2K_g +2 e^{2u_1^n}-\rho e^{u_1^n}|\psi_1^n|^2 + \rho e^{u_1^n}\Abracket{\psi_1^n, \varphi_1^n}\equiv \alpha_1^n \to 0, 
 \end{align}
 \begin{align}\label{eq:PS:u2}
  \frac{1}{3}\parenthesis{-\Delta u_1^n-2\Delta u_2^n}+ 2K_g +2 e^{2u_2^n}-\rho e^{u_2^n}|\psi_2^n|^2 + \rho e^{u_2^n}\Abracket{\psi_2^n, \varphi_2^n}\equiv \alpha_2^n \to 0
 \end{align}
in $ H^{-1}(M)$ and
 \begin{align}\label{eq:PS:psi}
  2\parenthesis{\D\psi_j^n-\rho e^{u_j^n}\psi_j^n }
  -\parenthesis{\D\varphi_j^n-\rho e^{u_j^n}\varphi_j^n}
  \equiv \beta_j^n \to 0 \qquad \mbox{ in } H^{-\frac{1}{2}}(\Sigma M), \quad j=1,2. 
 \end{align}

\

We divide the proof in two steps.

\medskip

\textbf{Step 1.}
First we show that this~$(PS)_c$ sequence~$(\bm{u}^n,\bm{\psi}^n)$ is uniformly bounded in~$H$. 
Compared to the super Liouville case~\cite{jevnikar2020existence}, there are some extra difficulties due to the coupling of the different fields, mainly the entanglement of~$u_1$ and~$u_2$.

\

Testing~\eqref{eq:PS:psi} against~$\varphi^n_j$ and  noting~\eqref{eq:PS:constraint}, we get 
\begin{align}
 -\int_M \Abracket{\D\varphi^n_j, \varphi_j^n} \dv_g +\rho \int_M e^{u^n_j} |\varphi_j^n|^2\dv_g = \parenthesis{\beta_j^n, \varphi_j^n}_{H^{-\frac{1}{2}}\times H^{\frac{1}{2}}}
 = o(\|\varphi_j^n\|_{H^{1/2}}).
\end{align}
Since we have assumed that~$\D$ has no kernel, this implies 
\begin{align}
 \|\varphi_j^n\|_{H^{1/2}}^2 + \int_M \rho e^{u_j^n}|\varphi_j^n|^2 \dv_g = o(\|\varphi_j^n\|_{H^{1/2}}), 
\end{align}
hence also 
\begin{align}
 \|\varphi_j^n\|_{H^{1/2}}=o(1), & & 
  \int_M \rho e^{u_j^n}|\varphi_j^n|^2 \dv_g=o(1). 
\end{align}

Meanwhile test~\eqref{eq:PS:psi} against~$\psi_j^n$, again using~\eqref{eq:PS:constraint}, to get
\begin{align}
 2\int_M \Abracket{\D\psi_j^n-\rho e^{u_j^n}\psi_j^n,\psi_j^n}\dv_g 
 = \parenthesis{\beta_j^n,\psi_j^n}_{H^{-1/2}\times H^{1/2}}
 = o(\|\psi_j^n\|_{H^{1/2}}). 
\end{align}
Then test~\eqref{eq:PS:u1} and~\eqref{eq:PS:u2} against the constant function~$1$:
\begin{align}
 \int_M 2K_g \dv_g 
 +\int_M  2e^{2u_1^n} -\rho e^{u_1^n}|\psi_1^n|^2 +\rho e^{u_1^n}\Abracket{\psi_1^n,\varphi_1^n}\dv_g = \parenthesis{\alpha_1^n,1}_{H^{-1}\times H^1}
 = o(1), 
\end{align}
\begin{align}
 \int_M 2K_g \dv_g 
 +\int_M  2e^{2u_2^n} -\rho e^{u_2^n}|\psi_2^n|^2 +\rho e^{u_2^n}\Abracket{\psi_2^n,\varphi_2^n}\dv_g = \parenthesis{\alpha_2^n,1}_{H^{-1}\times H^1}
 = o(1). 
\end{align}
For each~$j=1,2$, by H\"older inequality we have 
\begin{align}
 \left|\int_M \rho e^{u_j^n}\Abracket{\psi_j^n,\varphi_j^n}\dv_g\right|
 \le \eps \int_M \rho e^{u_j^n}|\psi_j^n|^2 \dv_g +\frac{\rho}{4\eps}\int_M e^{u_j^n}|\varphi_j^n|^2\dv_g. 
\end{align}
It follows that
\begin{multline}\label{eq:test u with const}
 o(1)+8\pi(\gamma-1)+(1-\eps)\int_M \rho e^{u_j^n}|\psi_j^n|^2\dv_g 
 \le \int_M 2 e^{2u_j^n}\dv_g \\
 \le o(1)+8\pi(\gamma-1)+(1+\eps)\int_M \rho e^{u_j^n}|\psi_j^n|^2\dv_g. 
\end{multline}

Now the level condition~\eqref{eq:PS:level condition} implies 
\begin{align}\label{eq:level estimate}
 c+8\pi(\gamma-1)+o(1)\ge \int_M \frac{1}{6}\parenthesis{|\nabla u_1^n|^2+|\nabla u_2^n|^2} \dv_g 
 &+\int_M 2K_g(u_1^n + u_2^n)\dv_g \\
 & +\int_M e^{2u_1^n}+ e^{2u_2^n}\dv_g  
 + o\parenthesis{\|\bm{\psi^n}\|_{H^{1/2}} }. 
\end{align}
Since~$K_g\equiv -1$, 
\begin{align}
 -2\int_M K_g (u_1^n+ u_2^n)\dv_g 
 = 2\int_M (u_1^n+ u_2^n)\dv_g  
 = 2\Vol(M,g) (\bar{u}_1^n+ \bar{u}_2^n)
 = 8\pi(\gamma-1)(\bar{u}_1^n+ \bar{u}_2^n).
\end{align}
Here we have used the Gauss--Bonnet formula
\begin{align}
 4\pi(1-\gamma)=\int_M K_g \dv_g=-\Vol(M,g). 
\end{align}
The Jensen's inequality then implies 
\begin{align}
 e^{2\bar{u}_1^n} + e^{2\bar{u}_2^n} 
 \le \fint_M e^{2u_1^n}+ e^{2u_2^n}\dv_g 
 \le \frac{c}{4\pi(\gamma-1)}+2+2(\bar{u}_1^n+ \bar{u}_2^n)
    +o(1)+ o\parenthesis{\|\bm{\psi^n}\|_{H^{1/2}}}.
\end{align}
By the convexity of the function~$t\mapsto e^t$, we have 
\begin{align}
 2e^{2(\bar{u}_1^n+\bar{u}_2^n)}
 \le e^{2\bar{u}_1^n} + e^{2\bar{u}_2^n} 
 \le \frac{c}{4\pi(\gamma-1)}+2+2(\bar{u}_1^n+ \bar{u}_2^n)
    +o(1)+ o\parenthesis{\|\bm{\psi^n}\|_{H^{1/2}}}.
\end{align}
This implies that 
\begin{align}
 |\bar{u}_1^n + \bar{u}_2^n| \le C(c,M)\parenthesis{1+o(1)+o\parenthesis{\|\bm{\psi^n}\|_{H^{1/2}} } }. 
\end{align}
 Here is a difference with respect to the super Liouville case~\cite{jevnikar2020existence}: we only get a control of~$|\bar{u}_1^n+ \bar{u}_2^n|$ but not each of~$|\bar{u}_j^n|$ for~$j=1,2$. 
The separate~$|\bar{u}_j^n|$ will be controlled in the end. 
Combining with~\eqref{eq:level estimate}, this in turn gives the following estimates:
\begin{align}
 \int_M |\nabla u_1|^2 + |\nabla u_2|^2 \dv_g 
 +\int_M e^{2u^n_1}+ e^{2u^n_2}\dv_g 
 \le C(c,M)\parenthesis{1+ o(1)+ o(\|\bm{\psi}\|_{H^{1/2}})}
\end{align}
and then from~\eqref{eq:test u with const} it follows that 
\begin{align}
 \int_M \rho e^{u_j^n}|\psi_j^n|^2 \dv_g 
 \le \frac{2}{1-\eps} \int_M e^{2u_j^n}\dv_g + C
 \le C(c,M)\parenthesis{1 +o(1)+ o(\|\bm{\psi}\|_{H^{1/2}}) }, \quad j=1,2.
\end{align}

\

Next we estimate the spinors. 
Write~$\psi_j^n=\psi_j^{n+} +\psi_j^{n-}$. 
For~$\psi_j^{n+}$: we test~\eqref{eq:PS:psi} against~$\psi_j^{n+}$ to get
\begin{align}
 C\|\psi_j^{n+}\|^2_{H^{1/2}}
 &\le  \int_M \Abracket{\D\psi_j^n,\psi_j^{n+}}\dv_g \\
 =&\int_M \rho e^{u_j^n}\Abracket{\psi_j^n,\psi_j^{n+} }\dv_g 
  +\int_M \Abracket{\D\varphi_j^n-\rho e^{u_j^n}\varphi_j^n, \psi_j^{n+}}\dv_g 
  +\parenthesis{\beta_j^n,\psi_j^{n+}} \\
 \le& \sqrt{\rho}\parenthesis{\int_M e^{2u_j^n}\dv_g }^{\frac{1}{4}} \parenthesis{\int_M \rho e^{u_j^n}|\psi_j^n|^2\dv_g }^{\frac{1}{2}}
 \parenthesis{\int_M |\psi_j^{n+}|^4 \dv_g }^{\frac{1}{4} } 
 + o\parenthesis{\|\psi_j^{n+}\|_{H^{1/2}}} \\
 \le& \parenthesis{ C(c,\rho,M)( 1+ o(1) + \|\psi_1^n\|_{H^{1/2}} +\|\psi_2^n\|_{H^{1/2}} )}^{\frac{3}{4}} \|\psi_j^{n+}\|_{H^{1/2}}
 + o\parenthesis{\|\psi_j^{n+}\|_{H^{1/2}}}. 
\end{align}
For~$\psi_j^{n-}$, by~\eqref{eq:PS:constraint} we have 
\begin{align}
 C\|\psi_j^{n-}\|^2_{H^{1/2}} 
 \le -\int_M \Abracket{\D\psi_j^n,\psi_j^{n-}}\dv_g 
 & =-\rho\int_M e^{u_j^n}\Abracket{\psi_j^n,\psi_j^{n-}}\dv_g \\ 
 \le& \parenthesis{ C(c,\rho,M)(1+ o(1) + \|\psi_1^n\|_{H^{1/2}} +\|\psi_2^n\|_{H^{1/2}}  ) }^{\frac{3}{4}} \|\psi_j^{n-}\|_{H^{1/2}}
\end{align}
summing these two estimates together we get 
\begin{align}
 C \parenthesis{\|\psi_j^{n+}\|^2_{H^{1/2}}+ \|\psi_j^{n-}\|^2_{H^{1/2}} }
  \le & \parenthesis{ C(c,\rho,M)(1+ o(1) + \|\bm{\psi}^n\|_{H^{1/2}} ) }^{\frac{3}{4}} \parenthesis{ \|\psi_j^{n+}\|_{H^{1/2}} +\|\psi_j^{n-}\|_{H^{1/2}} }  \\
 & + o\parenthesis{\|\psi_j^n\|_{H^{1/2}} }.  
\end{align}
From this it follows that 
\begin{align}
 \|\psi_1^n\|_{H^{1/2}}+\|\psi_2^n\|_{H^{1/2}}
 \le C(c,\rho, M). 
\end{align}
Consequently we also get 
\begin{align}
 \int_M |\nabla u_1^n|^2 +|\nabla u_2^n|^2\dv_g 
 +\int_M e^{2u_1^n} + e^{2u_2^n}\dv_g 
 \le C(c,\rho, M). 
\end{align}
In particular,
\begin{align}
 e^{2\bar{u}_j^n}
 \le \fint_M e^{2u_j^n}\dv_g 
 \le C(c,\rho, M), 
\end{align}
hence 
\begin{align}
 \bar{u}_j^n \le C(c,\rho,M), \qquad j=1,2.  
\end{align}
Recall that we already know~$|\bar{u}_1^n + \bar{u}_2^n|\le C(c,\rho,M)$.
They together implies that 
\begin{align}
 |\bar{u}_1^n| + |\bar{u}_2^n| \le C(c,\rho, M). 
\end{align}
Thus,
\begin{align}
 \|u_j^n\|^2_{H^1} = |\bar{u}_j^n|^2 + \|\nabla u_j^n\|^2_{L^2}
 \le C(c,\rho, M). 
\end{align}
Therefore, the~$(PS)_c$ sequence~$(\bm{u}^n, \bm{\psi}^n)$ is uniformly bounded in~$H$. 

\

\textbf{Step 2.}
We show that there is a subsequence which converges (strongly) in~$N_\rho$ to a constrained critical point for~$J_\rho|_{N_\rho}$. 
 
By the Banach--Alaoglu theorem, we can assume that the sequence~$(\bm{u}^n,\bm{\psi}^n)$ is weakly convergent in~$H$ to a weak limit~$(\bm{u}^\infty, \bm{\psi}^\infty)\in H$. 
By the Moser--Trudinger embedding theorem for the function components and the (fractional) Sobolev embedding theorems for the spinors, we know that 
\begin{align}
 e^{u_j^n} \to e^{u_j^\infty} \qquad \mbox{ strongly in } \quad L^p(M) ,\quad \forall \; p<\infty, 
\end{align}
\begin{align}
 \psi_j^n\to \psi_j^\infty, \qquad \mbox{ strongly in } \quad L^q, \quad \forall 1\le q<4. 
\end{align}
Consequently,~$(\bm{u}^\infty, \bm{\psi}^\infty)$ is a weak solution of~\eqref{eq:super Toda-1}, hence also a smooth solution by the regularity theory. 
In particular,~$(\bm{u}^\infty, \bm{\psi}^\infty)$ lies in~$N_\rho$. 

By using the equations~\eqref{eq:PS:u1},~\eqref{eq:PS:u2} and~\eqref{eq:PS:psi} we readily conclude that~$(\bm{u}^n,\bm{\psi}^n)$ actually converges strongly in~$H$ to the limit~$(\bm{u}^\infty,\bm{\psi}^\infty)$.  
In particular, since~$N_\rho$ is a submanifold of~$H$, the sequence~$(\bm{u}^n,\bm{\psi}^n)$ also converges in~$N_\rho$ to~$(\bm{u}^\infty,\bm{\psi}^\infty)$.

Therefore, the constrained functional~$J_\rho|_{N_\rho}$ satisfies the Palais-Smale condition (C), as desired. 
\end{proof}

It remains to find a nontrivial constrained critical point for~$J_\rho|_{N_\rho}$, which is the subject of next section.

\

\section{Min-max solutions in the Nehari manifold}

The advantage of constraining the functional to the Nehari manifold~$N_\rho$ is that the functional becomes less indefinite, at least locally near the origin~$(0,0)\in N_\rho$, so that the classical min-max strategy could be applied in searching for a nonzero solution. 
This will be done in the spirit of~\cite{jevnikar2020existence, jevnikar2021sinhgordon} highlighting the differences between the super Toda and the super Liouville cases. 

\

Consider the functionals~$F$ and~$Q_\rho$ defined by
\begin{align}
 F(\bm{u})\coloneqq
 & \int_M  \frac{1}{3}\parenthesis{|\nabla u_1|^2 +|\nabla u_2|^2+\Abracket{\nabla u_1, \nabla u_2}} +2 K_g (u_1 + u_2) + (e^{2u_1}+ e^{2u_2})+\int_M 2K_g \dv_g \\
 =&\int_M \frac{1}{3}\parenthesis{|\nabla u_1|^2 +|\nabla u_2|^2+\Abracket{\nabla u_1, \nabla u_2}} 
  +\parenthesis{ e^{2u_1}-1-2u_1 }
  +\parenthesis{ e^{2u_2}-1-2u_2 } \dv_g,
\end{align}
and 
\begin{align}
 Q_\rho(\bm{u}, \bm{\psi})
 \coloneqq & 
 \int_M \Abracket{\D\psi_1- \rho e^{u_1}\psi_1,\psi_1}
          +\Abracket{\D\psi_2-\rho e^{u_2}\psi_2,\psi_2} \dv_g.
\end{align}
Then~$J_\rho= F+ Q_\rho$. 
Note that~$J_\rho(0,0)=0$,~$F(\bm{u})\ge 0$ is convex, while~$Q_\rho(\bm{u}, \bm{\psi})$ is quadratic in~$\bm{\psi}$ but may be negative.

\medskip

In this section we will confine ourselves within the ball
\begin{align}
 N_\rho\cap B_R(0) =\braces{(\bm{u}, \bm{\psi}) \mid \|\bm{u}\|^2 + \|\bm{\psi}\|^2 <R^2}
\end{align}
for some small~$R>0$,
where~$\|\bm{u}\|^2 \equiv \|u_1\|_{H^1}^2 + \| u_2\|_{H^1}^2$ and~$\|\bm{\psi}\|^2 \equiv \|\psi_1\|_{H^{1/2}}^2 +\|\psi_2\|_{H^{1/2}}^2$.  
This~$R$ can be shrank later if needed. 
Note that inside the ball~$\{\|\bm{u}\|^2 \le R^2\}$, and 
\begin{align}
 \| e^{u_j}-1\|_{L^2} \le C(R)\|u_j\|^2_{H^1}, \qquad j=1,2. 
\end{align}

\

We analyze the local behavior of the functional near the origin~$(\bm{0}, \bm{0})\in N_\rho$. 
The constraints $G(\bm{u},\bm{\psi})=0$ implies that, for~$j\in \{1,2\}$, 
\begin{align}
 \int_M \Abracket{\D\psi_j-\rho e^{u_j}\psi_j, \psi_j^-}\dv_g =0. 
\end{align}
Consequently,
\begin{align}
  -\int_M \Abracket{\D\psi_j, \psi_j^-}\dv_g 
  =&  -\int_M \rho e^{u_j}\Abracket{\psi_j^+,\psi_j^-}\dv_g 
      -\int_M \rho e^{u_j}|\psi_j^-|^2 \dv_g \\
  \le & \,\rho \int_M e^{u_j} |\psi_j^+| |\psi_j^-| \dv_g.   
\end{align}
Thus we get 
\begin{align}
 \frac{\lambda_1}{1+\lambda_1}\|\psi_j^-\|_{H^{1/2}}^2 
 \le -\int_M \Abracket{\D\psi_j, \psi_j^-}\dv_g 
 \le \rho \|e^{u_j}\|_{L^2} \|\psi_j^+\|_{H^{1/2}} \|\psi_j^-\|_{H^{1/2}},
\end{align}
that is, 
\begin{align}\label{eq:negative part of spinors}
 \|\psi_j^-\|_{H^{1/2}} \le C(R,\lambda_1)\rho  \|\psi_j^+\|_{H^{1/2}}. 
\end{align}

Consider the functional 
\begin{align}
 J_\rho(\bm{u},\bm{\psi})
 =& F(\bm{u}) + Q_\rho(\bm{u},\bm{\psi}) \\
 =& F(\bm{u}) 
 +\sum_{j=1,2}\int_M \Abracket{(\D-\rho)\psi_j,\psi_j^+}\dv_g 
 +\sum_{j=1,2}\int_M \rho(1-e^{u_j})\Abracket{\psi_j,\psi_j^-}\dv_g \label{eq:splitting of the functional}
\end{align}
The first part,~$F(\bm{u})$, is coercive for~$\|\bm{u}\|\le R$ small: 
\begin{align}
 F(\bm{u})
 \ge& \int_M \frac{1}{6}|\nabla u_1|^2 + \frac{1}{6}|\nabla u_2|^2\dv_g 
    +\int_M (e^{2u_1}-1-2u_1) + (e^{2u_2}-1-2u_2)\dv_g \\
  \ge & \int_M \frac{1}{6}|\nabla u_1|^2 + \frac{1}{6}|\nabla u_2|^2\dv_g + C(|\bar{u}_1|^2 + |\bar{u}_2|^2)\\
  \ge & C\|\bm{u}\|^2, 
\end{align}
where we have used Jensen inequality to get 
\begin{align}
 \int_M e^{2u_j}-1-2u_j \dv_g \ge e^{2\bar{u}_j}-1-2\bar{u}_j
 \ge C\bar{u}_j^2
\end{align}
for~$|\bar{u}_j|<R$ sufficiently small. 
Moreover, the last part of~\eqref{eq:splitting of the functional} is of higher order due to~\eqref{eq:negative part of spinors}:
\begin{align}
 \int_M \rho (1-e^{u_j})\Abracket{\psi_j,\psi_j^+}\dv_g
 \le& C(R)\rho \| u_j\|_{H^1} \|\psi_j\|_{H^{1/2}} \|\psi_j^+\|_{H^{1/2}}. 
\end{align}
To deal with the middle part in~\eqref{eq:splitting of the functional}, we use the decomposition
\begin{align}
 \psi_j=\sum_{k\in \bZ_*} a_{j,k}\Psi_k, 
\end{align}
to see that 
\begin{align}
 \int_M \Abracket{(\D-\rho)\psi_j, \psi_j^+} \dv_g 
 = \sum_{k>0} (\lambda_k-\rho) a_{j,k}^2 
 = -\sum_{0<\lambda_k <\rho} (\rho-\lambda_k) a_{j,k}^2 
   +\sum_{\lambda_k>\rho} (\lambda_k-\rho) a_{j,k}^2
\end{align}
which is of saddle type: it is positive on the space~$H^{1/2,+}_a(\Sigma M)$ and negative on~$H^{1/2,+}_b(\Sigma M)$.  
Since~$H^{1/2,+}_b (\Sigma M)$ has finite dimension, we can apply the min-max strategy to obtain nontrivial solutions of saddle type.  

\medskip

We remark that the Nehari manifold~$N_\rho$ is a path-connected contractible space, which is actually simply connected.
For each fixed~$\bm{u}\in H^1(M)^2$, the space
\begin{align}
 N_{\rho,\bm{u}}
 \equiv \braces{\bm{\psi}\in H^{\frac{1}{2}}(\Sigma M)\mid G(\bm{u}, \bm{\psi})=0}
\end{align}
is a vector space, and the collection of these vector spaces gives a vector bundle over~$H^1(M)^2$, with total space~$N_\rho$. 

\medskip

Given~$\rho\notin\Spect(\D)$, recall the decomposition
\begin{align}
 H^{\frac{1}{2}}(\Sigma M)
 =H^{\frac{1}{2},+}_a(\Sigma M) 
 \oplus H^{\frac{1}{2},+}_b(\Sigma M)
 \oplus H^{\frac{1}{2},-}(\Sigma M), \quad  
 \psi=\psi_a^+ + \psi^+_b + \psi^-, 
\end{align}
and~$\dim H^{\frac{1}{2},+}_b(\Sigma M)= \#\braces{\lambda_k \in (0,\rho)}$. 
If~$\rho<\lambda_1$, then~$H^{\frac{1}{2},+}_b(\Sigma M)=\{0\}$. This is the easy case where the following argument produce a mountain pass critical point.

In general, consider the set 
\begin{align}
 \mathscr{N}_\rho\coloneqq \{\bm{0}\}\times H^{\frac{1}{2},+}_b(\Sigma M)
\end{align}
which is a linear subspace of~$N_\rho$, and~$J_\rho|_{\mathscr{N}_\rho}\le 0$. 
Actually, the tangent space of~$N_\rho$ at the trivial critical point~$(\bm{0},\bm{0})$ is 
\begin{align}
 T_{(\bm{0},\bm{0})} N_\rho 
 = (H^1(M))^2 \times (H^{\frac{1}{2},+}(\Sigma M))^2
\end{align}
and the Hessian~$\Hess(J_\rho)(\bm{0,\bm{0}})$ is a well-defined quadratic form, whose negative subspace is precisely~$\mathscr{N}_\rho$. 
As~$\dim \mathscr{N}_\rho$ is finite and changes when~$\rho$ crosses an eigenvalue of~$\D$, there should be bifurcation solutions near the eigenvalues. 
Here we will get nontrivial solutions for any~$\rho\notin\Spect(\D)$, which is a stronger result than the bifurcation phenomenon. 

\medskip

For a large enough~$\tau>0$ we consider then the cone centered at~$\mathscr{N}_\rho$: 
\begin{align}
 \mathcal{C}_\tau (\mathscr{N}_\rho)\coloneqq 
 \biggr\{(\bm{u},\bm{\psi}) \in N_\rho : 
 \|\bm{u}\|_{H^1}^1 + \|\bm{\psi}^-\|_{H^{1/2}}^2 + \|\bm{\psi}^+_a\|^2_{H^{1/2}} <\tau\|\bm{\psi}^+_b\|^2_{H^{1/2}} \biggr\}.
\end{align}
This cone deformation retracts to the space~$\mathscr{N}_\rho$.
Then it is not hard to prove
\begin{lemma}\label{lemma:mountain level}
 There exist~$\tau>1$,~$R<1$ and~$C>0$ such that
 \begin{align}
  J_\rho(\bm{u},\bm{\psi})
  \ge C\parenthesis{\|\bm{u}\|^2_{H^1} + \|\bm{\psi}\|^2_{H^{1/2}}}, \qquad
  \forall (\bm{u}, \bm{\psi})\in (N_\rho\cap B_R(\bm{0},\bm{0})) \setminus \mathcal{C}_\tau (\mathscr{N}_\rho)
 \end{align}

\end{lemma}
The computations are similar to the super Liouville case~\cite{jevnikar2020existence}, so we omit the details here. With the above choice, consider the set 
\begin{align}
 L_1\coloneqq 
 (N_\rho\cap \p B_R(\bm{0},\bm{0})) \setminus \mathcal{C}_\tau (\mathscr{N}_\rho) 
\end{align}
on which the functional is bounded from below by a positive constant. It is nonempty, and is homeomorphic to the collar around the set~$\p B_R(\bm{0},\bm{0})\cap (H^1(M))^2\times (H^{\frac{1}{2},+}_1(\Sigma M))^2$. 

\medskip

To apply the classical linking theorem (see e.g.~\cite{ambrosetti1995primer, struwe2008variational}), it remains to construct another subset~$L_2$ which links~$L_1$ such that~$J_\rho|_{L_2}\le 0$. 
The construction of~$L_2$ is the same as that in~\cite{jevnikar2020existence}, except that we have to extend the components of functions and spinors. 
The requirement that~$J_\rho|_{L_2}\le 0$ remains true since the functions in the set~$L_2$ are constants, along which the coupling feature for the two function components of the super Toda system is not present in the functional~$J_\rho$.
The set~$L_2$ is actually made of the boundary of a cylinder segment~$\mathcal{D}\subset N_\rho$, i.e.~$L_2=\p \mathcal{D}$, which, passing through the origin, winds around the set~$L_1$, see Figure 1 in~\cite{jevnikar2020existence}.
The set~$\mathcal{D}$ is compact.
Let
\begin{align}
 \Gamma\coloneqq \braces{\alpha\colon\mathcal{D}\to N_\rho \ : \ \alpha|_{\p\mathcal{D}}= \id_{\mathcal{\p\mathcal{D}}}}. 
\end{align}
This is nonempty since~$\id_{\mathcal{D}}\in\Gamma$.
Then it is standard to define the linking level 
\begin{align}
 c_1\coloneqq \inf_{\alpha\in\Gamma} \max_{x\in\mathcal{D}} J_\rho(\alpha(x))
\end{align}
which is a critical level, namely there exists a critical point~$(\bm{u}_*, \bm{\psi}_*)$ with~$J_\rho(\bm{u}_*, \bm{\psi}_*)=c_1$. 
From Lemma~\ref{lemma:mountain level} we see that~$c_1\ge CR^2>0$, hence the critical point~$(\bm{u}_*, \bm{\psi}_*)$ above is not the trivial one. 
This holds for any~$\rho\notin\Spect(\D)$, thus proves Theorem~\ref{thm:existence}. 

Note that if~$\rho<\lambda_1$, what we obtained is actually a solution of mountain pass type. 



\

\bibliographystyle{siam}
\bibliography{Super-Toda}

\end{document}